\documentclass[letterpaper, twoside, 12pt]{article}
\usepackage{amsmath,amsthm,amssymb,amsfonts}
\usepackage{enumitem}
\usepackage{graphicx}
\usepackage[small]{titlesec}
\usepackage{hyperref}
\usepackage{cite}
\usepackage{caption}
\usepackage{subcaption}
\usepackage{tikz}
\usepackage{float}
\usepackage{marginnote}

\newcommand{\N}{\mathbb{N}}
\newcommand{\Z}{\mathbb{Z}}

\newcommand{\R}{\mathbb{R}}

\newcommand{\z}{\mathcal{Z}}
\newcommand{\Row}{\mathcal{R}}
\newcommand{\Col}{\mathcal{C}}
\newcommand{\T}{\mathcal{T}}
\newcommand{\te}{\mathcal{T}_{\text{\rm en}}}
\newcommand{\M}{\mu}
\newcommand{\me}{\mu_{\text{\rm en}}}

\newcommand{\ten}{\tau_{\text{\rm en}}}
\newcommand{\Ne}{\mathcal{N}}
\newcommand{\col}{\text{\tt col}}
\newcommand{\row}{\text{\tt row}}
\newcommand{\vr}{\vec{r}}
\newcommand{\vc}{\vec{c}}

\newcommand{\bZ}{\Z}
\newcommand{\cT}{\mathcal{T}}
\newcommand{\cM}{\mathcal{M}}
\newcommand{\cA}{\mathcal{A}}
\newcommand{\Aen}{\mathcal{A}_{\text{\rm en}}}
\newcommand{\Mth}{\mu_{\text{\rm th}}}
\newcommand{\Men}{\mu_{\text{\rm en}}}
\newcommand{\cR}{\mathcal{R}}
\newcommand{\cC}{\mathcal{C}}

\newtheorem{theorem}{Theorem}[section]
\newtheorem{lemma}[theorem]{Lemma}
\newtheorem{proposition}[theorem]{Proposition}
\newtheorem{corollary}[theorem]{Corollary}

\theoremstyle{definition}

\begin{document}
\thispagestyle{empty}

\begin{center}\Large
{\bf Maximal spanning time for neighborhood growth on the Hamming plane}\footnote{Version 1, \today}
\end{center}
\begin{center}
{\sc Janko Gravner}\\
{\rm Mathematics Department}\\
{\rm University of California}\\
{\rm Davis, CA 95616, USA}\\
{\rm \tt gravner{@}math.ucdavis.edu}
\end{center}
\begin{center}
{\sc J.E. Paguyo}\\
{\rm Mathematics Department}\\
{\rm University of California}\\
{\rm Davis, CA 95616, USA}\\
{\rm \tt jepaguyo@ucdavis.edu}
\end{center} 
\begin{center}
{\sc Erik Slivken}\\
{\rm Mathematics Department}\\
{\rm University of California}\\
{\rm Davis, CA 95616, USA}\\
{\rm \tt erikslivken{@}math.ucdavis.edu}
\end{center} 


\begin{abstract}
\vspace*{-0.25em}
We consider a long-range growth dynamics on the two-dimensional integer lattice, initialized by a finite set of occupied points. 
Subsequently, a site $x$ becomes occupied if the pair consisting of the counts of occupied sites along the entire horizontal and vertical lines through $x$ lies outside a fixed Young diagram $\z$. 
We study the extremal quantity $\M(\z)$, the maximal finite time at which the lattice is fully occupied. 
We give an upper bound on $\M(\z)$ that is linear in the area of the bounding rectangle of $\z$, and a lower bound $\sqrt{s-1}$, where $s$ is the side length of the largest square contained in $\z$. 
We give more precise results for a restricted family of initial sets, and for a simplified version of the dynamics. 
\end{abstract}

\let\thefootnote\relax
\let\thefootnote\relax

\vspace*{-1cm}


\section{Introduction}

The {\em Hamming plane} 
is the Cartesian product of two complete graphs on $\Z_+$ and so it has vertex set $\Z_+^2$ with an edge between any pair of points that differ in a single coordinate. We refer to the points in $\Z_+^2$ as {\em sites}. 
Investigation of percolation and growth models on the Hamming plane and related highly connected graphs is a recent development \cite{Siv, GHPS, BBLN, Sli}, 
and this paper addresses an extremal quantity associated with a growth process introduced in \cite{GSS}.  

We keep the terminology and notation from \cite{GSS}. For $a,b \in \N$, we let $R_{a,b} = ([0,a-1] \times [0,b-1]) \cap \Z_+^2$ be the discrete $a \times b$ rectangle. 
A union $\z = \bigcup_{(a,b) \in \mathcal{I}} R_{a,b}$ of rectangles over some finite set $\mathcal{I} \subseteq \N^2$ is called a {\em zero-set}. Note that it is possible that $\z = \emptyset$. 
Also note that we restrict our consideration to finite zero-sets. 

For a site $x \in \Z_+^2$, we denote by $L^h(x)\subset \Z_+^2$ and $L^v(x)\subset \Z_+^2$ the horizontal and vertical lines through $x$, respectively. The {\em neighborhood} of $x$ then is $\Ne(x) = L^h(x) \cup L^v(x)$. 
The row and column counts of $x$ in a set $A\subset \Z_+^2$ are given by 
\begin{align*}
\row(x,A) = | L^h(x) \cap A| \text{ and } \col(x,A) = \left| L^v(x) \cap A \right|.
\end{align*}

A zero-set $\z$ determines a {\em neighborhood growth transformation} $\T : 2^{\Z_+^2} \to 2^{\Z_+^2}$ as follows. Fix $A \subseteq \Z_+^2$. If $x \in A$, then $x \in \T(A)$. If $x \notin A$, then $x \in \T(A)$ 
if and only if the pair of row and column counts of $x$ lies outside the zero-set, i.e., $(\row(x,A), \col(x,A)) \notin \z$. The {\em neighborhood growth dynamics} is given by the discrete time trajectory obtained by iteration of $\cT$: 
$A_t = \T^t(A)$, for $t \geq 0$. We call sites in $A_t$ {\em occupied} and sites in $A_t^c$ {\em empty}. See Figure \ref{FigNbhdGrowth} for an example of neighborhood growth dynamics.


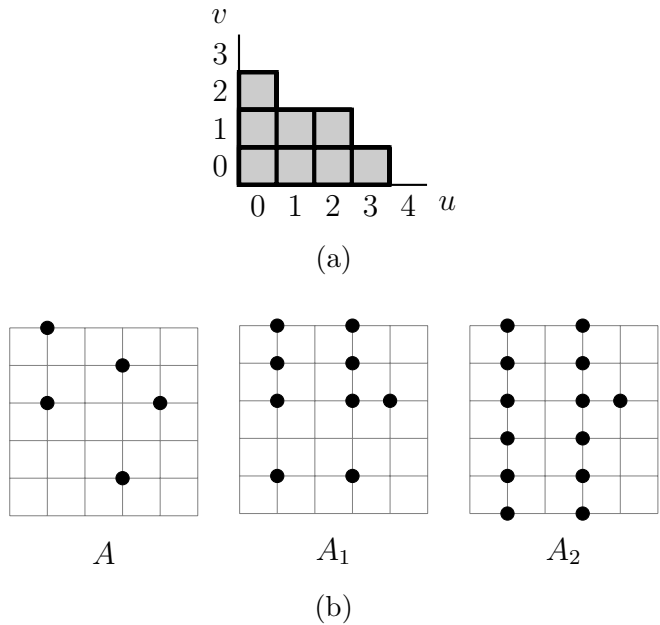
\begin{figure}[!t]
\centering

\begin{subfigure}{1\linewidth}
\centering
\begin{tikzpicture}[scale=0.50]
	\draw[thick,-] (0,0) -- (5,0) node[anchor=north west] {$u$};
	\draw[thick,-] (0,0) -- (0,4) node[anchor=south east] {$v$};
    	\draw (0.5 cm, 0pt) node (0.5 cm,0pt) [anchor=north] {$0$};
	\draw (1.5 cm, 0pt) node (1.5 cm,0pt) [anchor=north] {$1$};
    	\draw (2.5 cm, 0pt) node (2.5 cm,0pt) [anchor=north] {$2$};
    	\draw (3.5 cm, 0pt) node (3.5 cm,0pt) [anchor=north] {$3$};
	\draw (4.5 cm, 0pt) node (4.5 cm,0pt) [anchor=north] {$4$};
	\draw (0pt,0.5 cm) -- (0pt,0.5 cm) node [anchor=east] {$0$};
	\draw (0pt,1.5 cm) -- (0pt,1.5 cm) node [anchor=east] {$1$};
	\draw (0pt,2.5 cm) -- (0pt,2.5 cm) node [anchor=east] {$2$};
	\draw (0pt,3.5 cm) -- (0pt,3.5 cm) node [anchor=east] {$3$};
	\draw [ultra thick, draw=black, fill=black!20!white] (0,0) grid (1,3) rectangle (0,0);
	\draw [ultra thick, draw=black, fill=black!20!white] (0,0) grid (3,2) rectangle (0,0);
	\draw [ultra thick, draw=black, fill=black!20!white] (0,0) grid (4,1) rectangle (0,0);
\end{tikzpicture}
\caption{} 
\vspace{0.5cm}
\end{subfigure}
\begin{subfigure}{1\linewidth}
\centering
\captionsetup{justification=centering,margin=3cm}
\begin{tikzpicture}[scale=0.50]
	\draw[step=1cm,gray,very thin] (0,0) grid (5,5);
	\filldraw[black] (1,3) circle (5pt);
	\filldraw[black] (1,5) circle (5pt);
	\filldraw[black] (3,1) circle (5pt);
	\filldraw[black] (3,4) circle (5pt);
	\filldraw[black] (4,3) circle (5pt);
	\node at (2.5, -1) {$A$};
\end{tikzpicture}
\quad
\begin{tikzpicture}[scale=0.50]
	\draw[step=1cm,gray,very thin] (0,0) grid (5,5);
	\filldraw[black] (1,3) circle (5pt);
	\filldraw[black] (1,5) circle (5pt);
	\filldraw[black] (3,1) circle (5pt);
	\filldraw[black] (3,4) circle (5pt);
	\filldraw[black] (4,3) circle (5pt);
	\filldraw[black] (1,1) circle (5pt);
	\filldraw[black] (3,3) circle (5pt);
	\filldraw[black] (1,4) circle (5pt);
	\filldraw[black] (3,5) circle (5pt);
	\node at (2.5, -1) {$A_1$};
\end{tikzpicture}
\quad
\begin{tikzpicture}[scale=0.50]
	\draw[step=1cm,gray,very thin] (0,0) grid (5,5);
	\foreach \x in {0,1,2,3,4,5}
		\filldraw[black] (1,\x) circle (5pt);
	\foreach \y in {0,1,2,3,4,5}
		\filldraw[black] (3,\y) circle (5pt);
	\filldraw[black] (4,3) circle (5pt);
	\node at (2.5, -1) {$A_2$};
\end{tikzpicture}
\caption{}
\end{subfigure}
\caption{An example of neighborhood growth dynamics with zero-set given in (a) and the initial set in the leftmost panel of (b). The next two panels of (b) depict the subsequent two iterations. 
Observe that $\T^3(A) = \T^2(A)$, thus $A$ does not span.} \label{FigNbhdGrowth}
\end{figure}

The simplest example, {\em line growth}, introduced as {\em line percolation} in \cite{BBLN}, is given by a rectangular zero-set $\z = R_{a,b}$ for some $a,b \in \N$. 
Another special case, perhaps the most important one, is {\em threshold growth}, which is determined by an integer {\em threshold} $\theta \geq 1$. This natural 
growth rule is defined on an arbitrary graph as follows: a site $x$ becomes occupied when the number of already occupied sites among its neighbors is at least $\theta$.
This dynamics was first introduced on trees in \cite{CLR} and is typically called {\em bootstrap percolation}. 
The most common setting, with many deep and surprising results, is a graph of the form $[k]^\ell$, a Cartesian product of path graphs of $k$ points, and thus with standard nearest neighbor lattice connectivity \cite{AL,GG,Hol,BB,GHM,BBDM}.
On the Hamming plane, threshold growth is given by the triangular zero-set $\z = \{(u,v) : u + v \leq \theta - 1\}$; see \cite{GHPS} and \cite{GSS} for further background.  

Define the set of eventually occupied sites by $A_{\infty} = \T^{\infty}(A) = \bigcup_{t \geq 0} A_t$. The set $A$ {\em spans} if $A_{\infty} = \bZ_+^2$. 
For a fixed zero-set $\z$, we let $\mathcal{A} = \mathcal{A}(\z)$ denote the collection of all finite spanning subsets of $\Z_+^2$. It follows from Theorem 2.8 in \cite{GSS} that, for any $A \in \mathcal{A}$, the {\em spanning time} 
\begin{align*}
\tau(\z,A) = \min\{t\in \N: \T^t(A) = \bZ_+^2\}
\end{align*}
is finite. Our main focus of attention is the {\em maximal spanning time}, the extremal quantity defined by  
\begin{align*}
\M(\z) = \sup\{\tau(\z,A): A \in \mathcal{A} \}.
\end{align*}
Theorem 2.8 in \cite{GSS} shows that $\M(\z) < \infty$ by providing a very large upper bound (essentially the product of lengths of all rows of $\z$, multiplied by their number). 
One of our results is a substantial improvement of that bound (see Theorem \ref{generalUB} below). 
Before we give further definitions and state our results, we give a brief review of related results in the literature. 

Arguably, the topic that most closely resembles ours is the following classic problem on matrix powers. Let $\cM$ be the set of all primitive non-negative $n \times n$ matrices. 
Then let $h=h(n)$ be the smallest power that makes all elements of the power $A^h$ nonzero for every $A\in \cM$. 
See \cite{HV} for the solution of this problem, and related results for analogous extremal quantities obtained by replacing $\cM$ by some natural subsets of $\cM$.

Extremal quantities in growth models have been of substantial interest. Perhaps the most natural one is the smallest cardinality of a spanning set. 
It is a famous folk theorem that this quantity equals exactly $n$ for the bootstrap percolation on $[n]^2$ with $\theta=2$.   
For bootstrap percolation on $[n]^d$ with $\theta = 2$, the smallest spanning sets have size $\lceil d(n-1)/2\rceil + 1$ \cite{BBM}. 
Not much was known for thresholds $\theta \ge 3$ in such lattice setting until a recent breakthrough \cite{MN}. 
Smallest spanning sets have also been studied for bootstrap percolation on trees \cite{Rie} and on hypergraphs \cite{BBMR}.
For Hamming graphs, \cite{BBLN} determines the size of the smallest spanning sets for line growth, while \cite{GSS} gives bounds for general neighborhood growth. 

Maximal spanning time results are comparatively scarce. In \cite{BP2}, it is shown that the maximal spanning time for bootstrap percolation on $[n]^2$ with $\theta=2$ is $13n^2/18 + O(n)$ (see also \cite{BP1}). 
For bootstrap percolation on the hypercube $[2]^n$, the maximal spanning time is $\lfloor n^2/3\rfloor$ when $\theta=2$ \cite{Prz} and $n^{-1+o(1)}2^n$ when $\theta \geq 3$ \cite{Har}. 
A related clique-completion process is studied from this perspective in \cite{BPRS}. It appears that in general the maximal spanning time is considerably more complex than the smallest size of a spanning set. 
In our case, a major difficulty is non-monotonicity of $\mu$: if $\z \subseteq \z'$, then $\cA(\z') \subseteq \cA(\z)$ but $\tau(A,\z) \leq \tau(A,\z')$ for all $A \in \cA(\z')$, 
thus it is not clear how $\mu(\z)$ and $\mu(\z')$ compare.
As announced, our first result is an improved upper bound from \cite{GSS} on the maximal spanning time. 

\begin{theorem} \label{generalUB} 
For any $m,n \in \N$ and zero-set $\z \subseteq R_{m,n}$,
\begin{align*}
\M(\z)\leq 2mn+5.
\end{align*}
\end{theorem}

We are able to obtain a better upper bound for a special class of initial sets, which will also provide a lower bound for $\mu(\z)$. 

A finite set $A \subseteq \Z_+^2$ is {\em thin} if, for every site $x \in A$, either $\row(x,A) = 1$ or $\col(x,A) = 1$. That is, any point $x\in A$ has no other points of $A$ either on the horizontal line or on the vertical line through $x$. 
We define $\cA_{\text{\rm th}}$ to be the set of all spanning thin sets, and let
$$
\Mth(\z)=\max\{\tau(A,\z): A\in \cA_{\text{\rm th}}\}.
$$
Arguably, thin sets are the simplest general family of initial sets and are for this reason used in \cite{GHPS,GSS} and in Section \ref{specialcases}. 
In certain circumstances, thin set constructions are close to optimal \cite{GSS}; in the present context, the extent to which $\M(\z)$ and $\Mth(\z)$ are comparable remains unclear (see open problem 3 in Section \ref{openquestions}).   
The utility of thin sets in part comes from their connection to a simplified growth dynamics, which we now introduce. 

The row and column {\em enhancements} $\vec{r} = (r_0,r_1,\ldots) \in \Z_+^\infty$ and $\vec{c} = (c_0,c_1,\ldots) \in \Z_+^\infty$ are two weakly decreasing sequences of non-negative integers which increase row and column counts by fixed values. 
To be precise, the {\em enhanced neighborhood growth dynamics} \cite{GSS} is given by the triple $(\z,\vec{r},\vec{c})$, which defines a growth transformation $\T_{\text{\rm en}}: 2^{\Z_+^2} \to 2^{\Z_+^2}$ as follows:
\begin{align*}
\T_{\text{\rm en}}(A) = A \cup \{(i,j) \in \bZ_+^2 : (\row((i,j),A) + r_j, \col((i,j),A) + c_i) \notin \z \}.
\end{align*}
By default, we initialize the enhanced growth by the empty set, and we say that the pair of enhancements ($\vr, \vc)$ {\em spans} for $\z$ if $\bigcup_{t \geq 0} \T_{\text{\rm en}}^t(\emptyset)=\bZ_+^2$. 
We let $\mathcal{A}_{\text{\rm en}}$ be the set of all pairs of enhancements $(\vr,\vc)$ that have finite support and span for $\z$. Next, we introduce the enhanced spanning time
\begin{align*}
\ten(\z,\vr,\vc) &= \inf\{t \in \N : \te^t(\emptyset) = \bZ_+^2 \}, 
\end{align*}
and finally define the corresponding maximal quantity
\begin{align*}
\me(\z) = \max\{ \ten(\z,\vr,\vc) : (\vr,\vc) \in \mathcal{A}_{en} \}.
\end{align*}
We next state a comparison result between $\Mth$ and $\Men$. 

\begin{theorem} \label{thin-enhanced} 
For all zero-sets $\z$, 
\begin{align*}
\Men(\z)-2 \leq \Mth(\z)\le 2\Men(\z).
\end{align*}
\end{theorem} 

For a zero-set $\z$, we let $s=s(\z)$ to be the side of the largest square included in $\z$, that is, the integer $s\ge 0$ such that $R_{s,s}\subseteq\z$ but $R_{s+1,s+1}\not\subseteq \z$. 
The upper bound we obtain for $\Men$ and $\Mth$ is linear in $s$, and is in this sense the best possible.   

\begin{theorem} \label{enhancedUB}
For any zero-set $\z$, 
\begin{align*}
&\me(\z) \leq 4s+1,\\
&\Mth (\z) \leq 8s+2.
\end{align*}
Moreover,  
\begin{align*}
&\sup_{\z\ne\emptyset}\Men(\z)/s(\z)\in (0,\infty),\\
&\sup_{\z\ne\emptyset}\Mth(\z)/s(\z)\in (0,\infty). 
\end{align*}
\end{theorem} 

Finally, we give the lower bounds on the maximal spanning times. We do not know whether these are in any sense optimal (see open problem 2 in Section \ref{openquestions}).

\begin{theorem}\label{generalLB}
For any zero-set $\z \neq \emptyset$,   
\begin{align*}
\me(\z) \geq s^{1/2}
\end{align*}
and 
\begin{align*}
\M(\z) \geq \Mth(\z)\ge (s-1)^{1/2}.
\end{align*}
\end{theorem}

The paper is organized as follows. In Section \ref{preliminaries} we introduce additional notation and definitions, and prove preliminary results.
In Section \ref{specialcases}, we address special cases of neighborhood growth. 
We prove Theorem \ref{thin-enhanced} in Section \ref{enhthincomparison}, Theorems \ref{generalUB} and \ref{enhancedUB} in Section \ref{UpperBound}, and Theorem \ref{generalLB} in Section \ref{lowerbound}. 
Finally we conclude with a selection of open questions in Section \ref{openquestions}.



\section{Preliminaries}\label{preliminaries}

\subsection{Notation and Terminology}

We define the partial order $\preceq$ on $\Z_+^2$ as follows. 
For two sites $z=(i,j)$ and $z'=(i',j')$, $z\preceq z'$ if and only if $i \leq i'$ and $j\leq j'$.  

A {\em Young diagram} is then a set of sites $X \subseteq \Z^2_+$ such that $z' \in X$ and $z \preceq z'$ implies $z \in X$ 
for all sites $z,z' \in \Z_+^2$. Observe that any zero-set is a Young diagram. 

For $A \subseteq \Z_+^2$, we denote the respective projections of $A$ onto the $x$-axis and $y$-axis by $\pi_x(A)$ and $\pi_y(A)$. 

Consider a vector $\vr = (r_0,r_1,\ldots) \in \Z_+^\infty$ with weakly decreasing entries. The size of $\vr$ is $|\vr| = \sum_i r_i$. The support of $\vr$ is the smallest interval $[0,N-1]$ such that $r_i = 0$ for $i \geq N$. 
We will often write $\vr$ as a finite vector, omitting its zero coordinates. 

We say a set $B \subseteq \Z_+^2$ is {\em covered} at time $t$ by either regular or enhanced growth dynamics if every site of $B$ is occupied at time $t$ by the respective dynamics.  

A {\em line} in $\Z_+^2$ is either $L^h(x)$ (also called a {\em row}) or $L^v(x)$ (also called a {\em column}) for some $x \in \Z_+^2$.

The dynamics given by the growth transformation $\T$ is sometimes called the {\em regular dynamics} when it needs to be distinguished from the enhanced version. 

\subsection{Operations with Young diagrams}\label{YDoperations}

Let $X$ be a Young diagram and $k \in \N$. We define reductions of $X$ obtained by removing the $k$ leftmost columns or $k$ bottommost rows of $X$, 
\begin{align*}
& X^{\downarrow k} = \{(u,v-k) : (u,v) \in X, v \geq k\}, \\
& X^{\leftarrow k} = \{(u-k,v) : (u,v) \in X, u \geq k\},
\end{align*}
and the diagonal shift of $X$,
\begin{align*}
X^{\swarrow k} = X^{\leftarrow k \downarrow k}.
\end{align*}

Given two Young diagrams $X,Y$, we define the {\em infimal sum} of $X$ and $Y$ by
\begin{align*}
X \boxplus Y = (X^c + Y^c)^c
\end{align*}
where $X^c = \Z_+^2 \setminus X$. See Figure \ref{FigGridSum} for an example. For a Young diagram $X$, define its closure $\overline{X} = X \cup (\Z_+^2)^c$ and  
its height function $\phi_X : \R \to \R \cup \{\infty\}$ so that $\overline{X} + [0,1]^2 = \{(x,y) \in \R^2: y \leq \phi_X(x)\}$. 
Then the terminology comes from the fact that $\phi_{X \boxplus Y} = \phi_X \square \phi_Y$, where $\square$ is the infimal convolution (see for example Section 5 of \cite{Roc}). 

The following lemma in particular establishes that the set of Young diagrams equipped with the operation $\boxplus$ is a commutative monoid. We omit the routine proof. 


\begin{figure}[!t]
\centering

\centering
\begin{tikzpicture}[scale=0.50]
	\draw[thick,-] (0,0) -- (4,0) node[anchor=north west] {$u$};
	\draw[thick,-] (0,0) -- (0,4) node[anchor=south east] {$v$};
    	\draw (0.5 cm, 0pt) node (0.5 cm,0pt) [anchor=north] {$0$};
	\draw (1.5 cm, 0pt) node (1.5 cm,0pt) [anchor=north] {$1$};
    	\draw (2.5 cm, 0pt) node (2.5 cm,0pt) [anchor=north] {$2$};
	\draw (3.5 cm, 0pt) node (3.5 cm,0pt) [anchor=north] {$3$};
	\draw (0pt,0.5 cm) -- (0pt,0.5 cm) node [anchor=east] {$0$};
	\draw (0pt,1.5 cm) -- (0pt,1.5 cm) node [anchor=east] {$1$};
	\draw (0pt,2.5 cm) -- (0pt,2.5 cm) node [anchor=east] {$2$};
	\draw (0pt,3.5 cm) -- (0pt,3.5 cm) node [anchor=east] {$3$};
	\draw [ultra thick, draw=black, fill=black!20!white] (0,0) grid (2,2) rectangle (0,0);
	\draw [ultra thick, draw=black, fill=black!20!white] (0,0) grid (3,1) rectangle (0,0);
	\node at (2, -2) {$X$};
\end{tikzpicture}	
\qquad
\begin{tikzpicture}[scale=0.50]
	\draw[thick,-] (0,0) -- (4,0) node[anchor=north west] {$u$};
	\draw[thick,-] (0,0) -- (0,4) node[anchor=south east] {$v$};
    	\draw (0.5 cm, 0pt) node (0.5 cm,0pt) [anchor=north] {$0$};
	\draw (1.5 cm, 0pt) node (1.5 cm,0pt) [anchor=north] {$1$};
    	\draw (2.5 cm, 0pt) node (2.5 cm,0pt) [anchor=north] {$2$};
	\draw (3.5 cm, 0pt) node (3.5 cm,0pt) [anchor=north] {$3$};
	\draw (0pt,0.5 cm) -- (0pt,0.5 cm) node [anchor=east] {$0$};
	\draw (0pt,1.5 cm) -- (0pt,1.5 cm) node [anchor=east] {$1$};
	\draw (0pt,2.5 cm) -- (0pt,2.5 cm) node [anchor=east] {$2$};
	\draw (0pt,3.5 cm) -- (0pt,3.5 cm) node [anchor=east] {$3$};
	\draw [ultra thick, draw=black, fill=black!20!white] (0,0) grid (2,1) rectangle (0,0);
	\draw [ultra thick, draw=black, fill=black!20!white] (0,0) grid (1,3) rectangle (0,0);
	\node at (2, -2) {$Y$};
\end{tikzpicture}
\qquad
\begin{tikzpicture}[scale=0.50]
	\draw[thick,-] (0,0) -- (6,0) node[anchor=north west] {$u$};
	\draw[thick,-] (0,0) -- (0,6) node[anchor=south east] {$v$};
    	\draw (0.5 cm, 0pt) node (0.5 cm,0pt) [anchor=north] {$0$};
	\draw (1.5 cm, 0pt) node (1.5 cm,0pt) [anchor=north] {$1$};
    	\draw (2.5 cm, 0pt) node (2.5 cm,0pt) [anchor=north] {$2$};
	\draw (3.5 cm, 0pt) node (3.5 cm,0pt) [anchor=north] {$3$};
	\draw (4.5 cm, 0pt) node (4.5 cm,0pt) [anchor=north] {$4$};
	\draw (5.5 cm, 0pt) node (5.5 cm,0pt) [anchor=north] {$5$};
	\draw (0pt,0.5 cm) -- (0pt,0.5 cm) node [anchor=east] {$0$};
	\draw (0pt,1.5 cm) -- (0pt,1.5 cm) node [anchor=east] {$1$};
	\draw (0pt,2.5 cm) -- (0pt,2.5 cm) node [anchor=east] {$2$};
	\draw (0pt,3.5 cm) -- (0pt,3.5 cm) node [anchor=east] {$3$};
	\draw (0pt,4.5 cm) -- (0pt,4.5 cm) node [anchor=east] {$4$};
	\draw (0pt,5.5 cm) -- (0pt,5.5 cm) node [anchor=east] {$5$};
	\draw [ultra thick, draw=black, fill=black!20!white] (0,0) grid (1,5) rectangle (0,0);
	\draw [ultra thick, draw=black, fill=black!20!white] (0,0) grid (2,3) rectangle (0,0);
	\draw [ultra thick, draw=black, fill=black!20!white] (0,0) grid (4,2) rectangle (0,0);
	\draw [ultra thick, draw=black, fill=black!20!white] (0,0) grid (5,1) rectangle (0,0);
	\node at (3, -2) {$X \boxplus Y$};
\end{tikzpicture}

\caption{An example of the infimal sum of two Young diagrams.}\label{FigGridSum}
\end{figure}
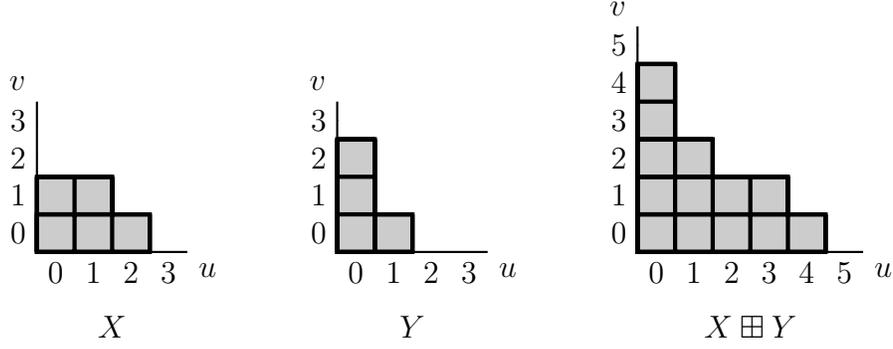


\begin{lemma}\label{gridsumprops}
Let $X,Y,Z$ be Young diagrams. The infimal sum has the following properties:
\begin{enumerate}
	\item \label{GSYD} $X \boxplus Y$ is a Young diagram.
	\item \label{identity} $X \boxplus \emptyset = X$.
	\item \label{associate} $(X \boxplus Y) \boxplus Z = X\boxplus (Y\boxplus Z)$.
	\item \label{gridcommute} $X\boxplus Y = Y \boxplus X$.
	\item \label{gridintersect} $(X \boxplus Y) \cap (Z\boxplus Y) = (X\cap Z) \boxplus Y$.
	\item \label{gridsumcontainment} If $X \subseteq Y$, then $X \boxplus Z \subseteq Y \boxplus Z$. 
\end{enumerate}
\end{lemma}

Assume $Y$ and $Z$ are Young diagrams. Let $\Delta(Z,Y)$ consist of all Young diagrams $X$ such that $Z \subseteq X \boxplus Y$. 
The {\em infimal difference} of $Z$ and $Y$ is defined as
\begin{align*}
Z \boxminus Y = \bigcap_{X \in \Delta(Z,Y)} X.
\end{align*}

\begin{lemma}\label{griddiffprops}
Let $Y,Z$ be Young diagrams. The infimal difference has the following properties:
\begin{enumerate}
\item \label{GDYD} $Z \boxminus Y \in \Delta(Z,Y)$.  
\item \label{griddiff2} $Z \boxminus Y \subseteq Z$. 
\end{enumerate}
\end{lemma}

\begin{proof}
Property \eqref{GDYD} holds since an intersection of Young diagrams is a Young diagram and by Lemma \ref{gridsumprops}\eqref{gridintersect}. Property \eqref{griddiff2} holds since $Z \in \Delta(Z,Y)$ implies $Z \boxminus Y \subseteq Z$. 
\end{proof}

\subsection{Enhanced growth}\label{enhgrowthsec}

Given a pair of enhancements $(\vec{r},\vec{c})$, we form a pair of Young diagrams $(\Row,\Col)$ such that the row counts of $\Row$ are given by $\vec{r}$ and the column counts of $\Col$ are given by $\vec{c}$. 
Therefore we use the two pairs interchangeably to describe an enhanced dynamics. The following lemma explains why enhanced growth is simpler than regular growth. 

\begin{lemma}\label{enhgrowthprops}
Let $(\Row,\Col)$ be enhancements that span for a zero-set $\z$. The set of occupied sites $A_t$ satisfies the following for all $t \geq 0$:
\begin{enumerate} 
\item The set $A_t$ is a Young diagram. \label{enhyoung}
\item The concave corners of $A_t$ must grow: if $(i-1,j),(i,j-1) \in \overline{A_t}$, then $(i,j) \in A_{t+1}$ for any $i,j \geq 0$. \label{enhconcave}
\item The sites with identical enhancements become occupied simultaneously: if $r_i = r_{i'}$ and $c_j = c_{j'}$, then $(i,j) \in A_t$ if and only if $(i',j') \in A_t$. \label{enhsametime}
\end{enumerate}
\end{lemma}

\begin{proof}
We prove Property \eqref{enhyoung} by induction. When $t = 0$, $A_0 = \emptyset$. 
Now suppose $A_t$ is a Young diagram for some $t \geq 0$. 
Assume $x = (i,j) \notin A_t$, $x' = (i',j') \notin A_t$, and $x \preceq x'$. 
As $A_t$ is a Young diagram, $\row(x,A_t) \geq \row(x',A_t)$ and $\col(x,A_t) \geq \col(x',A_t)$.
As $c_i \geq c_{i'}$ and $r_i \geq r_{j'}$, and $\z$ is a Young diagram,  
$(\row(x, A_t) + r_j, \col(x,A_t) + c_i) \notin \z$ implies $( \row(x', A_t) + r_{j'}, \col(x',A_t) + c_{i'} ) \notin \z$. Therefore $A_{t+1}$ is also a Young diagram. 

To prove Property \eqref{enhconcave}, suppose $x = (i,j) \notin A_{t+1}$. By Property \eqref{enhyoung}, $(i',j') \notin A_{t+1}$ for all $i' \geq i$ and $j' \geq j$. 
Thus $x' \notin A_s$ for all $s \geq t$ and $x \preceq x'$, which contradicts the fact that $(\Row,\Col)$ spans for $\z$. 

Let $x = (i,j)$ and $x' = (i',j')$. 
Then $(\row(x,A_{t-1}) + r_i, \col(x,A_{t-1}) + c_j) \notin \z$ if and only if $(\row(x',A_{t-1}) + r_{i'}, \col(x',A_{t-1}) + c_{j'}) \notin \z$. 
Therefore $(i,j) \in A_t$ if and only if $(i',j') \in A_t$. This proves property \eqref{enhsametime}. 
\end{proof}

\begin{lemma}\label{partition}
Fix a zero-set $\z$ and let $(\Row,\Col)$ be enhancements that span for $\z$. 
Let $\{I_i\}_{i = 0}^M \subseteq \Z_+$ be maximal intervals of equal column lengths of $\Col$, and similarly let $\{J_i\}_{i = 0}^N \subseteq \Z_+$ be maximal intervals of equal row lengths of $\Row$. Then
\begin{align*}
\bigcup_{i + j < t} I_i \times J_j \subseteq A_t
\end{align*}
for all $t \geq 0$. 
\end{lemma}

\begin{proof}
We use induction, beginning with the trivial base case $t = 0$. The inductive step follows from Lemma \ref{enhgrowthprops} \eqref{enhconcave} and \eqref{enhsametime}. 
\end{proof}

\begin{corollary}\label{partitiontimebound}
Let $\Row$ have $M$ nonzero row counts and $\Col$ have $N$ nonzero column counts. If $(\Row,\Col)$ span for $\z$, then $\ten(\z,\Row,\Col) \leq M+N +1$. 
\end{corollary}

The next lemma provides the key connection between the infimal sum and enhanced growth.

\begin{lemma}\label{gridsumspan}
Fix a zero-set $\z$. The enhancements $(\Row,\Col)$ span for $\z$ if and only if 
\begin{align*}
\z \subseteq \Row \boxplus \Col. 
\end{align*}
\end{lemma}

\begin{proof}
The pair of enhancements $(\Row,\Col)$ does not span if and only if there exists $(i,j) \notin \te^\infty(\emptyset)$ so that $(i-1,j)$ and $(i,j-1)$ are both in $\overline{\te^\infty(\emptyset)}$. 
For this to happen, we must have $(i+r_j,j+c_i) \in \z$. But $(i,c_i) \notin \Col$ and $(r_j,j) \notin \Row$, thus $\Row^c + \Col^c \not\subseteq \z^c$, and so $\z \not\subseteq \Row \boxplus \Col$. 

Conversely, if $\z \not\subseteq \Row \boxplus \Col$, there exist $(i,b) \notin \Row$ and $(a,j) \notin \Col$, so that $(i+a,j+b) \in \z$. 
Then $b \geq c_i$, $a \geq r_j$, and so $(i+r_j,j+c_i) \in \z$. Thus no point outside of $([0,i-1] \times [0,\infty)) \cup ([0,\infty) \times [0,j-1])$ becomes occupied, and consequently $(\Row,\Col)$ does not span. 
\end{proof}

\subsection{Perturbations of zero-sets} \label{zerosetperturbations}

Let $\tau_{\text{\rm line}}(\z,A)$ be the first time that the regular dynamics given by $\z$ covers a line in $\Z_+^2$. 
Define $$\mu_{\text{\rm line}}(\z) = \max_{A \in \mathcal{A}} \{\tau_{\text{\rm line}}(\z,A)\}.$$ 
We omit the simple proof of the following lemma. 

\begin{lemma}\label{haHAA}
For any zero-set $\z$,
\begin{align*}
\M(\z) &\leq \mu_{\text{\rm line}}(\z) + \max \left\{\M\left(\z^{\downarrow 1}\right),\M\left(\z^{\leftarrow 1}\right)\right\}.
\end{align*}
\end{lemma}

As already remarked, $\mu$ is not apparently monotone with respect to inclusion. 
We do however have a weaker form of monotonicity which is the subject of the next lemma.  

\begin{lemma} \label{shifty}
For any zero-set $\z \neq \emptyset$, 
\begin{align*}
&\M(\z) \geq \max \left\{\M\left(\z^{\downarrow 1}\right),\M\left(\z^{\leftarrow 1}\right)\right\}, \\
&\me(\z) \geq \max \left\{\me\left(\z^{\downarrow 1}\right),\me\left(\z^{\leftarrow 1}\right)\right\}. 
\end{align*}
\end{lemma}

\begin{proof}
To prove the first inequality, by symmetry we only need to show $\M(\z) \geq \M\left(\z^{\leftarrow 1}\right)$. 
Assume $A \in \mathcal{A}(\z^{\leftarrow 1})$. Let $A' = A \cup (\{(M,0)\} \times [0,N])$.
If $M$ and $N$ are large enough, then $A' \in \mathcal{A}(\z)$ and $\tau(\z^{\leftarrow 1},A) = \tau(\z,A') \leq \M(\z)$. 
Therefore $\M(\z^{\leftarrow 1}) \leq \M(\z)$. 

For the second inequality, we again only need to show $\Men(\z) \geq \Men\left( \z^{\leftarrow 1} \right)$. 
Assume $(\vr,\vc) \in \Aen(\z^{\leftarrow 1})$, where $\vr = (r_0,\ldots,r_{N_0})$ and $\vc = (c_0,\ldots,c_{M_0})$. Let 
\begin{align*}
\vr\,' = (r_0+1,\ldots,r_{N_0}+1,1,\ldots,1,0,\ldots) 
\end{align*}
be row enhancements with finite support $[0,N]$. 
For $N$ are large enough, $(\vr\,',\vc) \in \Aen(\z)$, and $\ten(\z^{\leftarrow 1},\vr,\vc) = \ten(\z,\vr\,',\vc) \leq \Men(\z)$. Thus $\M(\z^{\leftarrow 1}) \leq \Men(\z)$. 
\end{proof}

The next lemma gives a converse inequality. 

\begin{lemma} \label{then2}
For any zero-set $\z$,
\begin{align*}
\Men(\z)\le \Men(\z^{\swarrow 1})+2.
\end{align*}
\end{lemma}

\begin{proof}
Assume that $\cR$ and $\cC$, with respective number of rows and columns $N_0$ and  $M_0$, span for the zero-set $\z$. 
Then $\cR^{\leftarrow 1}$ and $\cC^{\downarrow 1}$ span for $\z^{\swarrow 1}$ and do so in at most $\Men(\z^{\swarrow 1})$ steps. 
The enhanced dynamics given by $(\z,\cR,\cC)$ and $(\z^{\swarrow 1}, \cR^{\leftarrow 1},\cC^{\downarrow 1})$ agree on the rectangle $[0,M_0-1]\times [0, N_0-1]$ by Lemma \ref{enhgrowthprops}. 
Therefore, the dynamics given by $(\z,\cR,\cC)$ covers this rectangle by time $\Men(\z^{\swarrow 1})$, and then needs at most two additional steps to fully occupy $\bZ_+^2$. 
The inequality follows. 
\end{proof}

\subsection{Thin sets}

As we will see in the next lemma, it is advantageous to 
permute rows and columns to arrange sites in a thin set 
in a certain manner reminiscent of convexity 
(see Figure \ref{FigSA}). We formalize this arrangement next.

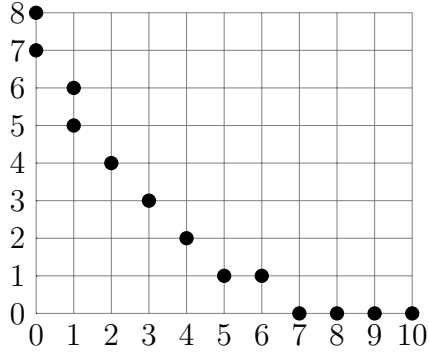
\begin{figure}[!t]
\centering
\captionsetup{justification=centering,margin=2cm}

\begin{tikzpicture}[scale=0.50]
	
	\draw[step=1cm,gray,very thin] (0,0) grid (10,8);
	\foreach \x in {0,1,2,3,4,5,6,7,8,9,10}
		\draw (\x cm, 0pt) node (\x cm,0pt) [anchor=north] {$\x$};
	\foreach \y in {0,1,2,3,4,5,6,7,8}
		\draw (0pt,\y cm) -- (0pt,\y cm) node [anchor=east] {$\y$};
	
	\foreach \x in {7,8,9,10}
		\filldraw[black] (\x,0) circle (5pt);
	\foreach \x in {5,6}
		\filldraw[black] (\x,1) circle (5pt);
	
	\filldraw[black] (4,2) circle (5pt);
	\filldraw[black] (3,3) circle (5pt);
	\filldraw[black] (2,4) circle (5pt);
	
	\foreach \y in {7,8}
		\filldraw[black] (0,\y) circle (5pt);
	\foreach \y in {5,6}
		\filldraw[black] (1,\y) circle (5pt);
		
\end{tikzpicture}	

\caption{A thin set in the standard arrangement with $\vec{r} = (4,2)$, $\vec{c}=(2,2)$, and $w=3$.} \label{FigSA}
\end{figure}

Two sets $A_1,A_2\subseteq \bZ_+^2$ are {\em equivalent} if there are permutations of rows and columns of $\bZ_+^2$ that map $A_1$ to $A_2$. It is clear that the spanning times of two equivalent sets are the same. 

An equivalence class of thin sets is given by finite (possibly empty) weakly decreasing vectors $\vec{r}$ and $\vec{c}$ with integer entries of at least $2$, which specify the number of occupied sites in the rows and columns that contain at least two sites, 
and a number $w \geq 0$ of isolated occupied sites. We now identify a specific representative of this equivalence class.

We say that a thin set $A$ is in the {\em standard arrangement} (see Figure \ref{FigSA}) if the following hold:
\begin{itemize}
\item The row counts $\row((0,j),A)$, $j\ge 0$, and column counts $\col((i,0),A)$, $i\ge 0$, are weakly decreasing.
\item If $x, x'\in A$ and $x \preceq x'$, then $x$ and $x'$ are on the same line.   
\end{itemize}

To achieve the standard arrangement, let $N_0$ (resp., $M_0$) be the number of entries of $\vr$ (resp., $\vc$), and consider the rectangle $R_{M_0+w+|\vr|, N_0+w+|\vc|}$. 
The sites in $A$ comprise, in order, the following diagonally adjacent intervals connecting the top left corner of this rectangle with its bottom right corner: 
vertical intervals of $c_i$ sites, $i=0,\ldots, N_0-1$, followed by $w$ single sites, and followed by horizontal intervals of $r_i$ sites, $i=N_0-1,\ldots, 0$ (see Figure \ref{FigSA}). 
It is straightforward to see that the standard arrangement is unique. 

\begin{lemma}\label{thingrowthprops}
Let $A$ be a thin set in the standard arrangement. Assume $Y\subset \bZ_+^2$ is a Young diagram. Let $A_0=A\cup Y$ and $A_1=\cT(A_0)$. Then 
\begin{enumerate}
\item \label{rowcolmonotone} If $x \preceq x'$, then for all $t\geq 0$, $\row(x,A_0) \geq \row(x',A_0)$ and $\col(x,A_0) \geq \col(x',A_0)$.  
\item \label{growthtype} If $x\preceq x'$ and $x'\in A_1\setminus A$, then $x\in A_1$.
\item \label{shapetype} The set $A_1$ is the union of a Young diagram and $A$.  
\end{enumerate}
Moreover, $A_t=\cT^t(A_0)$ is the union of a Young diagram and $A$ for all $t\ge 0$. 
\end{lemma}

\begin{proof}
To prove \eqref{rowcolmonotone}, it is, by symmetry, enough to prove the inequality for the row counts. Let $x=(i,j)$ and $x'=(i',j')$, where $i\le i'$ and $j\le j'$. 
Let $I'=\{a\in \bZ_+: (a,j')\in Y\}$, $I=\{a\in \bZ_+: (a,j)\in Y\}$, $J'=\{a\in \bZ_+: (a,j')\in A\}$, and $J=\{a\in \bZ_+: (a,j)\in A\}$. 
As $Y$ is a Young diagram, $I'\subseteq I$. As $A$ is in the standard arrangement, $|J'\setminus I|\le |J\setminus I|$. Therefore, 
\begin{align*}
\row(x',A_0) &= |I'\cup J'| \\
&= |I'|+|(J'\cap I)\setminus I'|+|J'\setminus I| \\
&\le |I'|+|I\setminus I'|+|J\setminus I|\\&=|I\cup J| \\
&= \row(x,A_0).
\end{align*}
This establishes \eqref{rowcolmonotone}, which immediately implies \eqref{growthtype}. Then \eqref{shapetype} follows, as $A_1=A\cup Y_1$, where 
$$Y_1=\bigcup_{x'\in A_1\setminus A}\{x:x\preceq x'\}$$ 
is a Young diagram. Finally, the last claim follows by induction. 
\end{proof}


\section{Special cases}\label{specialcases}

In this section we prove results on $\M(\z)$ for two special cases of neighborhood growth.
First is line growth, where $\z$ is a single rectangle. The second is {\em L-growth}, where $\z$ is a union of two rectangles: $\z = R_{a,b} \cup R_{c,d}$, where $a,b,c,d \in \N$ such that $a>c$ and $d>b$. 
It turns out that the bounds for L-growth do not depend on the larger numbers $a$ and $d$. 

\begin{lemma}\label{pastnbhd}
Let $\z = R_{m,n}$ and let $A \subseteq \Z_+^2$. If $L^h(x) \subseteq \T^2(A) \setminus \T(A)$, then there exists at least one site $y \in L^h(x)$ such that $L^v(y) \subseteq \T(A) \setminus A$. 
Similarly, if $L^v(x) \subseteq \T^2(A) \setminus \T(A)$, then there exists at least one site $y \in L^v(x)$ such that $L^h(y) \subseteq \T(A) \setminus A$. 
\end{lemma}

\begin{proof}
By symmetry we may assume that $L^h(x) \subseteq \T^2(A) \setminus \T(A)$. Then $\row(x,\T(A)) \geq m$ and $\row(x,A) < m$. There exists at least one site $y \in L^h(x)$ such that $y \in \T(A) \setminus A$. 
Since $\row(y,A) < m$, we must have that $\col(y,A) \geq n$. This implies $L^v(y) \subseteq \T(A) \setminus A$.
\end{proof} 

\begin{proposition} \label{rectmn}
If $\z = R_{m,n}$, then with $m \neq n$, then 
\begin{align*}
\M(\z) = \Mth(\z) =\begin{cases} 2\min\{m,n\}, & \text{if $m \neq n$}, \\ 2n-1, & \text{if $m=n$}. \end{cases}
\end{align*}
\end{proposition}

\begin{proof}
We only address the case when $m > n$, as the case $m=n$ is similar. 
Let $A$ be a spanning set for $\z$. By Lemma \ref{pastnbhd}, there are at least $n-1$ covered rows and $n-1$ covered columns at time $t = 2n-2$. 
At $t = 2n-1$, every column containing at least one site that lies outside of the $n-1$ spanned rows and $n-1$ spanned columns becomes covered. As $A$ spans, there must be at least $m$ covered columns at $t = 2n-1$.  
Therefore $\tau(\z,A) \leq 2n$. 

For the lower bound let $A$ be a thin set in the standard arrangement given by $\vr = (m-1,m-2,\ldots,2)$, $\vc = (n,n-1,\ldots,2)$, and $w = 2$. It is straightforward to check that $A$ spans in $2n$ steps. 
\end{proof} 

\begin{lemma}\label{everytwo}
Let $\z = R_{a,b} \cup R_{c,d}$ and let $A \subseteq \Z_+^2$. In every two time steps, at least one line is covered. 
\end{lemma}

\begin{proof}
Suppose $x \in \T^2(A) \setminus \T(A)$.  There exists $y \in \Ne(x)$ such that $y \in \T(A) \setminus A$. Suppose $y \in L^h(x)$. 
Then either $\col(y,A) \geq d$, or $b \leq \col(y,A) < d$ and $c \leq \row(y,A) < a$. 

If $\col(y,A) \geq d$, then column $L^v(y)$ is covered in $\T(A)$.  
Otherwise, $b \leq \col(y,A) < d$ and $c \leq \row(y,A) = \row(x,A) < a$, and $\col(x,A) < b$.
There exists $z \in L^v(x)$ such that $z \in \T(A) \setminus A$. Then $\col(x,A) = \col(z,A) < b$ which implies that $\row(z,A) \geq a$. Thus row $L^h(z)$ is covered in $\T(A)$. 
\end{proof}

\begin{proposition}\label{Lbound}
If $\z = R_{a,b} \cup R_{c,d}$, then 
\begin{align*}
2\min\{b,c\} \leq \M(\z) \leq 2(b+c)
\end{align*}
\end{proposition}

\begin{proof}
Without loss of generality assume $b \leq c$. By Lemma \ref{everytwo} at least one row or column is spanned in every two steps, and Lemma \ref{haHAA} gives us
\begin{align*}
\M(\z) \leq 2 + \max\left\{M\left(\z^{\leftarrow 1}\right), M\left(\z^{\downarrow 1}\right)\right\},
\end{align*}
and the upper bound follows by induction. 

For the lower bound, we consider two cases. If $d-b > c$, then by Lemma \ref{shifty}, $\M(\z) \geq \M(\z^{\downarrow b}) = 2c$.
Otherwise $d-b \leq c$. Let $\z' = \z^{\leftarrow c-b}$. Let $A$ be a thin set in the standard arrangement given by $\vr = (n-1,n-2,\ldots,2)$, $\vc = (b,b-1,\ldots,2)$, and $w=2$, where $n = \max\{a-c+b,d\}$. 
One can check that $\tau(\z',A) = 2b+1$, if $a-c+b \geq d$, and $\tau(\z',A) = 2b$, otherwise. 
Therefore by Lemma \ref{shifty}, $\M(\z) \geq \M(\z') \geq 2b$. 
\end{proof}


\section{Enhanced growth vs. growth from thin sets} \label{enhthincomparison}

In this section we prove Theorem \ref{thin-enhanced}. 

\begin{lemma}\label{enhancedineq}
For any zero-set $\z$, 
\begin{align*}
\me\left(\z^{\swarrow 1} \right) \leq \Mth(\z).
\end{align*}
\end{lemma}

\begin{proof} 
Let $(\Row,\Col)$ be some enhancements that span for $\z^{\swarrow 1}$, given respectively by infinite vectors $\vec{r} = (r_0,r_1,\ldots)$ and 
$\vec{c} = (c_0,c_1,\ldots)$ of respective finite supports $[0,N_{0}-1]$ and $[0,M_0-1]$. 
Form infinite vectors $\vec{r} \, ^+ = (r_i^+ : i \geq 0) = (r_0+1,r_1+1,\ldots,r_{N_0-1}+1,1,\ldots)$ and 
$\vec{c}\, ^+ = (c_j^+: j \geq 0) = (c_0+1,c_1+1,\ldots,c_{M_0-1}+1,1,\ldots)$. 

Observe that the enhancements $\vec{r}\,^+,\vec{c}\,^+$ span for $\z$. In fact, the enhanced dynamics with zero-set $\z^{\swarrow 1}$ and 
enhancements $\vec{r},\vec{c}$ and the enhanced dynamics with zero-set $\z$ and enhancements $\vec{r}\,^+,\vec{c}\,^+$ 
have the same occupied set $B_t^+$ at any time $t \geq 0$. It is important to note that when $B_t^+$ covers the rectangle 
$R_{M_0+1,N_0+1}$, it in fact fully occupies $\Z_+^2$. 

Choose $N$ large enough so that the square $S = R_{N,N}$ satisfies $R_{M_0+1,N_0+1} \subseteq S$ and $\z \subseteq S$. 
For $i < N$ let $r_i' = r_i^+$ and $c_i' = c_i^+$, and for $i \geq N$ let $r_i' = c_i' = 0$. 
Let $B_t$ be the set of occupied sites at time $t$ under the enhanced dynamics given by $(\z,\vec{r}\,',\vec{c}\,')$. 
By Lemma \ref{enhgrowthprops}, $B_t \cap S = B_t^+ \cap S$ for all $t \geq 0$ and therefore the enhancements $\vec{r}\,',\vec{c}\,'$ span for $\z$. 

Define a thin set $A$ in the standard arrangement given by the row vector $(r_0',\ldots,r_{N_0-1}')$, the column vector $(c_0',\ldots,c_{M_0-1}')$, and $w = N - N_0 + N - M_0 \geq 2$.
Observe that $S \cap A = \emptyset$. Let $A_t$ be the set of occupied sites at time $t$ starting from $A$ under the regular dynamics given by $\z$. 

We claim that $A_t \cap S = B_t \cap S$ for all $t \geq 0$. We use induction to prove this claim, which clearly holds at $t = 0$. 
Assume that the claim holds for some $t \geq 0$. Let $x \in S \setminus A_t = S \setminus B_t$. 
By Lemma \ref{thingrowthprops}, $\mathcal{N}(x) \cap A_t \cap A^c = \mathcal{N}(x) \cap (A_t \cap S)$ and again by 
Lemma \ref{enhgrowthprops}, $\mathcal{N}(x) \cap B_t = \mathcal{N}(x) \cap (B_t \cap S)$. 
Therefore $x \in A_{t+1}$ if and only if $x \in B_{t+1}$ which proves the induction step. 

When the enhanced dynamics given by $(\z,\vec{r}\,',\vec{c}\,')$ covers $S$, the enhanced dynamics given by $(\z,\vec{r}\,^+,\vec{c}\,^+)$ also covers $S$, and thus fully occupies $\Z_+^2$. 
Therefore, by the claim in the previous paragraph, $\ten(\z,\vec{r}\,^+,\vec{c}\,^+) \leq \tau(\z,A) \leq \Mth(\z)$, 
and then $\ten(\z^{\swarrow 1},\vec{r},\vec{c}) = \ten(\z,\vec{r}\,^+,\vec{c}\,^+) \leq \Mth(\z)$. 
As $\vec{r}$ and $\vec{c}$ are arbitrary enhancements that span for $\z^{\swarrow 1}$, the proof is concluded. 
\end{proof}

\begin{lemma}\label{youngspantime}
For any Young diagram $Y\in \mathcal{A}$, $\tau(\z,Y) \leq \me(\z)$.
\end{lemma}

\begin{proof}
The enhanced spanning time for the pair of enhancements $(\z,\emptyset)$ is given by the number of distinct row counts (including 0) in $\z$.  We also have $\tau(\z,\z) \leq \te(\z,\z,\emptyset)\leq \me(\z)$.  By Lemma 3.1 in \cite{GSS}, the Young diagram $Y \in \mathcal{A}$ if and only if $\z\subseteq Y$.  Therefore by monotonicity $\tau(\z,Y) \leq \tau(\z,\z) \leq \me(\z)$	
\end{proof}

\begin{lemma} \label{thinUbound}	
For any $\z$, $\Mth(\z) \leq 2\me(\z).$
\end{lemma}

\begin{proof}  
Let $A\in \mathcal{A}$ be a thin set in the standard arrangement.  Let $X$ be the set of points $x\in  \mathbb{Z}_+^2\setminus A$ such that $x\preceq x'$ for some $x'\in A$, that is, $X$ is the set of points strictly below $A$ (see Figure \ref{FigSALemma}).  
Let $\ell = \inf \{t\geq 0: X\subseteq A_t\}$ be the first time at which the regular dynamics starting from $A$ covers $X$.

In order to obtain an upper bound on $\ell$ we consider enhanced growth with enhancements $\Row$ and $\Col$ that are given by the row and column counts of $A$.  
Thus, for any $x=(i,j)\in \mathbb{Z}_+^2$ the row and column enhancements for $x$ are given by $r_j= \row(x,A)$ and $c_i=\col(x,A)$.  The pair $(\Row,\Col)$ spans for $\z$ under the enhanced dynamics as the proof for the two-Y construction in Lemma 3.3 of \cite{GSS} applies.  

For $t\geq 0$, let $B_t$ denote the set of occupied sites under enhanced growth given by $(\z,\Row,\Col).$  First we claim that, for $t\geq 0$, $A_t\cap X = B_t \cap X$.  We proceed by induction.  
The claim is true for $t=0$ as $A_0 \cap X = \emptyset = B_0\cap X$.  Suppose the claim is true for some $t\geq 0$.  
By Lemma \ref{enhgrowthprops} \eqref{enhyoung} and Lemma \ref{thingrowthprops} \eqref{shapetype}, for $x\in X\setminus A_t = X\setminus B_t,$
$$ 
\Ne(x) \cap A_t \cap A^c = \Ne(x) \cap B_t.	
$$
Therefore $\row(x,A_t) = \row(x,B_t) + \row(x,A)$ and the analogous equality holds for the column counts.  It follows that $A_{t+1} \cap X = B_{t+1} \cap X,$ which establishes the inductive claim. 

By the inductive claim $\ell \leq \me(\z)$.  Now, $A\cup X$ is a Young diagram that includes $A$, and therefore spans, and is also covered by $A_\ell$.  
By Lemma \ref{youngspantime}, $\tau(\z, A_\ell) \leq \me(\z)$ and therefore $\tau(\z,A) \leq \ell + \me(\z) \leq 2\me(\z)$.   
\end{proof}


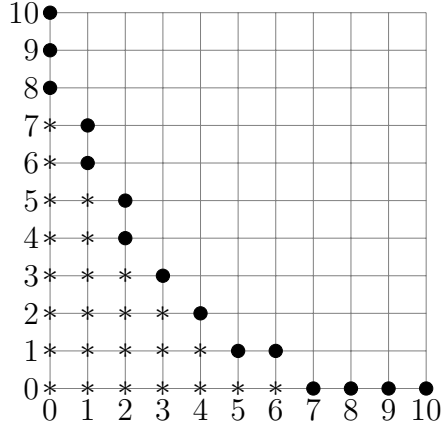
\begin{figure}[!t]
\centering

\begin{tikzpicture}[scale=0.50]
	
	\draw[step=1cm,gray,very thin] (0,0) grid (10,10);
	\foreach \x in {0,1,2,3,4,5,6,7,8,9,10}
		\draw (\x cm, 0pt) node (\x cm,0pt) [anchor=north] {$\x$};
	\foreach \y in {0,1,2,3,4,5,6,7,8,9,10}
		\draw (0pt,\y cm) -- (0pt,\y cm) node [anchor=east] {$\y$};
		
	\foreach \x in {7,8,9,10}
		\filldraw[black] (\x,0) circle (5pt);
	\foreach \x in {5,6}
		\filldraw[black] (\x,1) circle (5pt);
	
	\filldraw[black] (3,3) circle (5pt);
	\filldraw[black] (4,2) circle (5pt);
	
	\foreach \y in {8,9,10}
		\filldraw[black] (0,\y) circle (5pt);
	\foreach \y in {6,7}
		\filldraw[black] (1,\y) circle (5pt);
	\foreach \y in {4,5}
		\filldraw[black] (2,\y) circle (5pt);

	\foreach \x in {0,1,2,3,4,5,6}
		\draw[black] (\x,0) node {$\ast$};
	\foreach \x in {0,1,2,3,4}
		\draw[black] (\x,1) node {$\ast$};
	\foreach \x in {0,1,2,3}
		\draw[black] (\x,2) node {$\ast$};
	\foreach \x in {0,1,2}
		\draw[black] (\x,3) node {$\ast$};
	\foreach \x in {0,1}
		\draw[black] (\x,4) node {$\ast$};
	\foreach \x in {0,1}
		\draw[black] (\x,5) node {$\ast$};
	\foreach \y in {6,7}
		\draw[black] (0,\y) node {$\ast$};
			
\end{tikzpicture}

\caption{Illustration of the proof of Lemma \ref{thinUbound}. The thin set $A$ is marked by solid circles, and the set $X$ by asterisks.} \label{FigSALemma}
\end{figure}

\begin{proof}[Proof of Theorem \ref{thin-enhanced}]
The upper bound on $\Mth$ follows from Lemma \ref{thinUbound}, and the lower bound follows from Lemmas \ref{then2} and \ref{enhancedineq}. 
\end{proof}


\section{Upper bounds}\label{UpperBound}

In this section we prove Theorem \ref{generalUB} and Theorem \ref{enhancedUB}. 

\begin{lemma}\label{crosspoint}
Let $A$ be an initial set of occupied points and $A_t$ the resulting set of occupied points at time $t$ in the regular dynamics with zero-set $\z$. 
Assume that $x \in A_t \setminus A$ and $y \in A_t \setminus A$ are not neighbors. Let $w$ be the sole element of $L^h(x)\cap L^v(y)$ and $z$ the sole element of $L^v(x)\cap L^h(y)$.  
Then either $w \in A_t$ or $z \in A_t$.  
\end{lemma}

\begin{proof}
Let $(u_x,v_x)$ and $(u_y,v_y)$ denote the row and column counts of $x$ and $y$, respectively, in $A_{t-1}$. Since $x$ and $y$ are both occupied by time $t$, $(u_x,v_x) \notin \z$ and $(u_y,v_y) \notin \z$. 
The pair of row and column counts of $w$ (resp., $z$) in $A_{t-1}$ is given by $(u_x,v_y)$ (resp., $(u_y,v_x)$). 

Assume $u_x \leq u_y$ and $v_x \leq v_y$. Then $(u_x,v_y) \notin \z$ and $(u_y,v_x) \notin \z$, and consequently $w$ and $z$ are both occupied at time $t$. The same conclusion holds if $u_y \leq u_x$ and $v_y \leq v_x$. 
Otherwise suppose without loss of generality that $u_x \leq u_y$ and $v_y \leq v_x$. Then we claim $(u_y,v_x) \notin \z$. If not, then $(u_x,v_x) \in \z$ and $(u_y,v_y) \in \z$, a contradiction. Thus $z \in A_t$. 
\end{proof}

\begin{proof}[Proof of Theorem \ref{generalUB}]
Suppose that $\z \subseteq R_{m,n}$. Fix a set $A$ of initially occupied sites that spans.  
Also fix a time $k>0$ to be specified later. Suppose that $A_k\backslash A_{k-1}\neq \emptyset$ so that $\M(\z) \geq k$.  

For $i>0$, let $B_i = A_i\backslash A_{i-1}$ be the set of sites that become occupied at time $i$.  A horizontal line with at least $m$ sites in $A_i$ or a vertical line with at least $n$ sites in $A_i$ is said to be \emph{saturated} at time $i$.  
Let $I$ denote the set of times $i\in [1,k]$ such that no line becomes saturated or covered at time $i$.  We claim that $|[1,k]\setminus I| \leq 2m+2n-1.$  

To prove the claim, we make a few observations. If a horizontal or a vertical line is saturated at time $i$, then it becomes covered by time $i+1$.  
Once $n+m-1$ lines become saturated (resp., covered), at least $n$ horizontal lines or at least $m$ vertical lines will be saturated (resp., covered). 
If at least $n+m-1$ lines are saturated at time $i$, then at least $n+m-1$ lines are covered by time $i+1$ and all of $\Z_+^2$ is occupied by time $i+2.$  
At time $k-1$, $\Z_+^2$ is not covered, so there are at most $2(n+m-2)$ times in $[1,k-3]$ in which a line becomes either saturated or covered, which implies the claim.    

Let $B = \cup_{i\in I} B_i$.  If follows from the claim that 
\begin{equation}\label{generalUBeq0}
|B| \geq |I| \geq k-2m-2n + 1.
\end{equation}

On any given horizontal (resp., vertical) line, there are less than $m$ (resp., $n$) sites in $B$; otherwise, that line would become saturated by the last point added.  
Fix a $z\in B$ and suppose neither $L^h(z)$ nor $L^v(z)$ is saturated by time $k$.  Define the sets 
$$B_h = \left(L^h(z)\cap B\right ) \bigcup \{ z' \in B\backslash \Ne(z): L^h(z)\cap L^v(z') \in A_k \}$$
and 
$$B_v = B\backslash B_h.$$
The set $B_v$ consists of sites $z'$ that satisfy one of the following two properties: 
either $z' \in L^v(z) \backslash\{z\}$; or both $z' \in B \backslash \Ne(z)$ and $L^h(z)\cap L^v(z') \notin A_k$, in which case   
by Lemma \ref{crosspoint}, $L^v(z) \cap L^h(z') \in A_k$. 
Therefore $|\pi_x(B_h)| \leq |L^h(z) \cap A_k|$  and $|\pi_y(B_v)| \leq |L^v(z) \cap A_k|.$

For any subset $B'\subseteq B$, we have 
$$|B'| \leq (n-1)|\pi_x(B')|, \qquad |B'| \leq (m-1)|\pi_y(B')|.$$  
Therefore, the inequalities 
\begin{equation}\label{generalUBeq1}
|\pi_x(B_h)| \leq \frac{k-2m-2n}{2(n-1)}, \qquad |\pi_y(B_v)| \leq \frac{k-2m-2n}{2(m-1)}
\end{equation}
cannot both be satisfied, as otherwise
$$|B| = |B_h| + |B_v| \leq (n-1)|\pi_x(B_h)| + (m-1)|\pi_y(B_v)| \leq k-2m-2n.$$

Assume $k \geq 2(m-1)(n-1) + 2m + 2n$. For any $z\in B$, we claim a line through $z$ is covered at time $k+1$. If either $L^h(z)$ or $L^v(z)$ is saturated the claim is true.  
Otherwise one of the inequalities in \eqref{generalUBeq1} is not satisfied, and then either $|\pi_x(B_h)| \geq m$ or $|\pi_y(B_v)| \geq n$.  
In the former case $|L^h(z)\cap A_k| \geq |\pi_x(B_h)| \geq m$ and $L^h(z)$ is covered by time $k+1$. 
In the latter case $|L^v(z)\cap A_k| \geq |\pi_y(B_v)| \geq n$ and $L^v(z)$ is covered by time $k+1$.  

Let $H$ (resp., $V$) be the set of all $z\in B$ such that $L^h(z)$ (resp., $L^v(z)$) is covered at time $k+1$. 
By the previous paragraph, $B = H \cup V$. 
Let $k = 2(m-1)(n-1) + 2m + 2n + 1 = 2mn + 3$. 
For this choice of $k$, $|B|\geq 2(m-1)(n-1) + 2$ by \eqref{generalUBeq0}, 
and therefore either $|H| \geq (m-1)(n-1)+1$ or $|V| \geq (m-1)(n-1)+1$.  
As a horizontal line can contain at most $m-1$ sites in $H$, $|H| \geq (m-1)(n-1)+1$ implies that there are at least $n$ horizontal lines with sites in $H$ which are covered by time $k+1$. 
Similarly, as a vertical line can contain at most $n-1$ sites in $V$, $|V|\geq (m-1)(n-1)+1$ implies that there are at least $m$ vertical lines with sites in $V$ which are covered by time $k+1$. 
In either case, $\Z_+^2$ is fully occupied by time $k+2 = 2mn + 5$.     
\end{proof}

We now proceed to prove the better upper bound for enhanced growth. 
Let $\widetilde{\mathcal{A}}_{\rm en}$ consist of pairs $(\Row,\Col) \in \Aen$ such that $\Row \subseteq \z$ and $\Col \subseteq \z$. 

\begin{lemma}\label{containment}
For any zero-set $\z$, 
\begin{equation*}
\me(\z) = \sup\{\ten(\z,\Row,\Col) : (\Row,\Col) \in \widetilde{\mathcal{A}}_{\rm en}\}	.
\end{equation*}
\end{lemma}

\begin{proof}
Fix $(\Row,\Col)\in \mathcal{A}$ and let $\Row' = \z \boxminus \Col$ and $\Col' = \z \boxminus \Row'$. By Lemma \ref{griddiffprops}, $(\Row',\Col')$ span for $\z$, and $(\Row',\Col') \in \widetilde{\mathcal{A}}_{\rm en}$. 
Moreover, $\Row'\subseteq \Row$ and $\Col'\subseteq \Col$, and so $\ten(\z,\Row',\Col') \geq \ten(\z,\Row,\Col)$.
This shows that $\me(\z) \leq \sup\{\ten(\z,\Row,\Col) : (\Row,\Col) \in \widetilde{\mathcal{A}}_{\rm en}\}$, and the other inequality is trivial. 
\end{proof}

\begin{proof}[Proof of Theorem \ref{enhancedUB}]
The zero-set $\z$ contains at most $s$ rows and $s$ columns that are longer than $s$, so the number of distinct non-zero row counts of $\z$ is at most $2s$, and the same upper bound holds for the column counts of $\z$. 
For every $(\Row,\Col) \in \widetilde{\mathcal{A}}_{\rm en}$, $\Row \subseteq \z$ and $\Col \subseteq \z$ and so the number of distinct row counts of $\Row$, including zero, is at most $2s + 1$, 
and again the same upper bound holds for the number of distinct nonzero column counts of $\Col$. 
By Corollary \ref{partitiontimebound}, $A_{2s + 2s + 1} = \Z_+^2$. Thus, for every $(\Row,\Col) \in \widetilde{\mathcal{A}}_{\rm en}$,  
$\ten(\z,\Row,\Col) \leq 2s + 2s + 1 = 4s + 1$
and therefore by Lemma \ref{containment}, $\me(\z) \leq 4s + 1$. 
The upper bound on $\Mth$ follows from Theorem \ref{thin-enhanced}.
The two statements on the suprema over $\z$ now follow from Proposition \ref{rectmn}. 
\end{proof}



\section{Lower bounds} \label{lowerbound}

In this section we prove Theorem \ref{generalLB}. 

For $a,b,k \in \N$, define the $(a,b)$-staircase of size $k$ as the following union of rectangles:
\begin{align*}
S_{a,b,k} = \bigcup_{j=1}^k R_{bj,a(k+1-j)} = \left\{(i,j) : \left\lfloor \frac{i}{b} \right\rfloor + \left\lfloor \frac{j}{a} \right\rfloor \leq k-1 \right\}.
\end{align*}
This Young diagram has row counts of length $b, 2b, \cdots, kb,$ each with multiplicity $a$, and column counts $a, 2a, \cdots, ka$ each with multiplicity $b$. 

\begin{lemma} \label{staircasesumlemma}
Fix $a,b \in \N$. For any $k \geq 1$, 
$S_{a,b,k}$ lies strictly below the line $ax+by = (k+1)ab$ and contains the set of sites $\{(i,j) : ai+bj < kab\}$
Furthermore, if $k_1,k_2 \in \N$, 
\begin{equation} \label{staircasesum}
S_{a,b,k_1} \boxplus S_{a,b,k_2} = S_{a,b,k_1 + k_2}.
\end{equation}    
\end{lemma}

\begin{proof}
This is a straightforward verification. 
\end{proof}

The next lemma provides a general method for establishing lower bounds; see Figure \ref{FigSlopeLine} for an illustration. 

\begin{lemma} \label{ratslope}
Fix a zero-set $\z$ and two positive integers $a$ and $b$. Let $k$ be the smallest positive integer such that $ai + bj < kab$ for all $(i,j) \in \z.$  Let $(i_0,j_0)\in \z$ be a point that satisfies $(k-1)ab \leq ai_0+bj_0$. Then
\begin{align*}	
\me(\z) \geq \min\left(\left\lceil \frac{i_0+1}{b}\right \rceil, \left\lceil \frac{j_0+1}{a}\right \rceil \right).
\end{align*}
\end{lemma}

\begin{proof}
By Lemma \ref{staircasesumlemma}, $\z \subseteq S_{a,b,k}$. 
Assume $k_1,k_2 \in \N$ such that $k_1 + k_2 \geq k$. 
By \eqref{staircasesum} and Lemma \ref{gridsumspan}, 
the row and column enhancements $\Row = S_{a,b,k_1}$ and $\Col = S_{a,b,k_2}$ span for $\z$. 
 
Let $m= i_0 +1$ and $n=j_0+1$. Then $R_{m,n} \subseteq \z$. Moreover, $k_1 = \lceil{\frac{m}{b}}\rceil$ and $k_2 = \lceil{\frac{n}{a}}\rceil$ satisfy $k_1+k_2 \geq k$. 

In this paragraph we consider the dynamics with zero-set $R_{m,n}$. As $R_{m,n} \subseteq \z$, $\Row$ and $\Col$ span for $R_{m,n}$.  
Let $\ell = \min (k_1,k_2)$. At any time $t \leq \ell$, exactly $a$ new rows and $b$ new columns become covered. Therefore, $\Z_+^2$ does not become occupied by time $\ell$.

By monotonicity, $\ten(\z,\Row,\Col) \geq \ten(R_{m,n},\Row,\Col)$, and therefore $\me(\z) \geq \ell$. 
\end{proof}


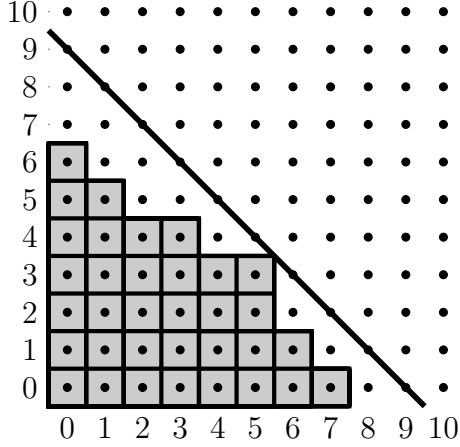
\begin{figure}[!t]
\centering
\captionsetup{justification=centering,margin=2cm}

\begin{tikzpicture}[scale=0.50]
    	\draw (0.5 cm, 0pt) node (0.5 cm,0pt) [anchor=north] {$0$};
	\draw (1.5 cm, 0pt) node (1.5 cm,0pt) [anchor=north] {$1$};
    	\draw (2.5 cm, 0pt) node (2.5 cm,0pt) [anchor=north] {$2$};
    	\draw (3.5 cm, 0pt) node (3.5 cm,0pt) [anchor=north] {$3$};
	\draw (4.5 cm, 0pt) node (4.5 cm,0pt) [anchor=north] {$4$};
	\draw (5.5 cm, 0pt) node (5.5 cm,0pt) [anchor=north] {$5$};
	\draw (6.5 cm, 0pt) node (6.5 cm,0pt) [anchor=north] {$6$};
	\draw (7.5 cm, 0pt) node (7.5 cm,0pt) [anchor=north] {$7$};
	\draw (8.5 cm, 0pt) node (8.5 cm,0pt) [anchor=north] {$8$};
	\draw (9.5 cm, 0pt) node (9.5 cm,0pt) [anchor=north] {$9$};
	\draw (10.5 cm, 0pt) node (10.5 cm,0pt) [anchor=north] {$10$};
	\draw (0pt,0.5 cm) -- (0pt,0.5 cm) node [anchor=east] {$0$};
	\draw (0pt,1.5 cm) -- (0pt,1.5 cm) node [anchor=east] {$1$};
	\draw (0pt,2.5 cm) -- (0pt,2.5 cm) node [anchor=east] {$2$};
	\draw (0pt,3.5 cm) -- (0pt,3.5 cm) node [anchor=east] {$3$};
	\draw (0pt,4.5 cm) -- (0pt,4.5 cm) node [anchor=east] {$4$};
	\draw (0pt,5.5 cm) -- (0pt,5.5 cm) node [anchor=east] {$5$};
	\draw (0pt,6.5 cm) -- (0pt,6.5 cm) node [anchor=east] {$6$};
	\draw (0pt,7.5 cm) -- (0pt,7.5 cm) node [anchor=east] {$7$};
	\draw (0pt,8.5 cm) -- (0pt,8.5 cm) node [anchor=east] {$8$};
	\draw (0pt,9.5 cm) -- (0pt,9.5 cm) node [anchor=east] {$9$};
	\draw (0pt,10.5 cm) -- (0pt,10.5 cm) node [anchor=east] {$10$};
	\draw [ultra thick, draw=black, fill=black!20!white] (0,0) grid (6,4) rectangle (0,0);
	\draw [ultra thick, draw=black, fill=black!20!white] (0,0) grid (8,1) rectangle (0,0);
	\draw [ultra thick, draw=black, fill=black!20!white] (0,0) grid (7,2) rectangle (0,0);
	\draw [ultra thick, draw=black, fill=black!20!white] (0,0) grid (4,5) rectangle (0,0);
	\draw [ultra thick, draw=black, fill=black!20!white] (0,0) grid (2,6) rectangle (0,0);
	\draw [ultra thick, draw=black, fill=black!20!white] (0,0) grid (1,7) rectangle (0,0);
	\draw[line width = 0.075cm] (0,10)--(10,0);

	\foreach \x in {0,1,2,3,4,5,6,7,8,9,10}
		\foreach \y in {0,1,2,3,4,5,6,7,8,9,10}
			\filldraw[black] (\y + 0.5,\x + 0.5) circle (3pt);

\end{tikzpicture}

\caption{Illustration of Lemma \ref{ratslope}, with $a=b=1$, $k = 9$, and $(i_0,j_0) = (5,3)$.}  \label{FigSlopeLine}
\end{figure}

\begin{proof}[Proof of Theorem \ref{generalLB}]
Let $k_r$ (resp., $k_c$) denote the largest number of rows (resp., columns) in $\z$ that are all of the same length and are of length at least $s$.  
If $k_r < s^{1/2}$, then there are at least $s^{1/2}$ rows of different lengths in $\z$. Let $(\Row,\Col)$ be enhancements such that $\Row = \z$ and $\Col = \emptyset$. 
Then rows with different lengths are spanned at different times, which gives a spanning time of at least $s^{1/2}$. 
Similarly, if $k_c < s^{1/2}$, $\Row = \emptyset$ and $\Col = \z$ induce a spanning time of at least $s^{1/2}$. 

Otherwise, $k_r \geq s^{1/2}$ and $k_c \geq s^{1/2}$. Let $d_c$ be the number of columns with length strictly larger than the common length of the $k_c$ columns.  
Analogously let $d_r$ be the number of rows with length strictly larger than that of the $k_r$ rows.
Define $\z' = \z^{\leftarrow d_c\downarrow d_r}$. We use Lemma \ref{ratslope} with $a=b=1$ and zero-set $\z'$.  
There exists $(i_0,j_0)\in \z'$ such that $i_0+1\geq k_c$, $j_0+1\geq k_r$ and the line with slope $-1$ through $(i_0+1,j_0+1)$ does not intersect $\z'$.  By Lemma \ref{ratslope} $\me(\z') \geq \min\{i_0+1,j_0+1\} \geq \min(k_c,k_r)\geq s^{1/2}$. 
By Lemma \ref{shifty}, $\me(\z') \leq \me(\z)$, so $\me(\z) \geq s^{1/2}.$ This proves the inequality for $\Men$. 

The remainder of the theorem follows from Lemma \ref{enhancedineq} and the observation that if $R_{s,s} \subseteq \z$, then $R_{s-1,s-1} \subseteq \z^{\swarrow 1}$.
\end{proof}



\section{Open questions}\label{openquestions}

\begin{enumerate}

\item Is there an explicit formula for $\M(\z)$ or for $\Men(\z)$ for threshold growth $\z =\{(u,v): u+v \leq \theta-1\}$? We know that $\theta + 1 \leq \M(\z) \leq 2\theta^2 + 5$. 

\item Let $\alpha_c^{\text{\rm min}}$ be the infimum of the the set of exponents $\alpha>0$ for which 
$$
\inf_{\z} \frac{\Men(\z)}{s^\alpha}=0.
$$
What is the value of $\alpha_c^{\text{\rm min}}$? We suspect that the answer is $1/2$ 
and know that $1/2 \leq \alpha_c^{\text{\rm min}}\leq \alpha_c^{\text{\rm max}} = 1$ (from Theorems \ref{enhancedUB} and \ref{generalLB}). 
Here, $\alpha_c^{\text{\rm max}}$ is defined to be the supremum of the exponents $\alpha>0$ for which 
$$
\sup_{\z} \frac{\Men(\z)}{s^\alpha}=\infty. 
$$

\item Is there a finite exponent $\alpha>0$ so that 
$$
\sup_{\z} \frac{\M(\z)}{s^\alpha}<\infty?
$$

\item Call a set $A\subseteq \Z_+^2$  {\em $k$-thin} if, for every $x\in A$, either $\row(x,A)\le k$ or $\col(x,A)\le k$. 
Observe that $1$-thin sets are exactly the thin ones. Let $\Mth^k(\z)$ be the maximal finite spanning time of a $k$-thin set. Do there exist finite constants $c_k,C_k>0$ such that 
$$ c_k\,\Mth(\z)\le \Mth^k(\z)\le C_k\,\Mth(\z)?$$

\end{enumerate}



\section*{Acknowledgements} \label{acknowledgements}

Janko Gravner was partially supported by the NSF grant DMS-1513340, Simons Foundation Award \#281309, and the Republic of Slovenia's Ministry of Science program P1-285.
J.E. Paguyo was partially supported by the same NSF grant and the UC Davis REU program. 



\end{document}